\documentclass[leqno]{amsart}%
\usepackage{amsmath}
\usepackage{amsfonts}
\usepackage{amssymb}
\usepackage{hyperref}
\usepackage{graphicx}%
\providecommand{\U}[1]{\protect\rule{.1in}{.1in}}
\newtheorem{theorem}{Theorem}

\newtheorem{corollary}[theorem]{Corollary}

\newtheorem{definition}[theorem]{Definition}

\newtheorem{lemma}[theorem]{Lemma}

\newtheorem{proposition}[theorem]{Proposition}
\newtheorem{remark}[theorem]{Remark}

\newcommand{\ric}{\operatorname{Ric}}
\newcommand{\scal}{\operatorname{Scal}}
\renewcommand{\H}{\mathbb{H}}

\newcommand{\scaln}{\scal^{(n)}}

\newcommand{\sect}{\operatorname{Sect}}

\newcommand{\vol}{\operatorname{Vol}}

\newcommand{\ri}[8]{R_{x_0}(S^{#1}_{#2},S^{#3}_{#4},S^{#5}_{#6},S^{#7}_{#8})}
\newcommand{\ka}[4]{\kappa\left(S^{#1}_{#2},S^{#3}_{#4}\right)}

\renewcommand{\S}{{\mathbb{S}}}

\newcommand{\R}{\mathbb{R}}

\newcommand{\abs}[1]{\left|#1\right|}

\newcommand{\ton}[1]{\left(#1\right)}
\newcommand{\qua}[1]{\left[#1\right]}

\newcommand{\co}{{\mathcal O}}

\newcommand{\e}{{\varepsilon}}

\newcommand{\summ}[2]{{\sum_{#1\leq k<j\leq #2}}}
\newcommand{\xt}{\mathrm{xt}_{n+1}}
\newcommand{\xtt}{\mathrm{xt}_{3}}

\begin{document}
\setcounter{tocdepth}{1}
\title{Scalar curvature via local extent}
\author{Giona Veronelli}
\address{Universit\'e Paris 13, Sorbonne Paris Cit\'e, LAGA, CNRS ( UMR 7539)
99\\
avenue Jean-Baptiste Cl\'ement F-93430 Villetaneuse - FRANCE} \email{veronelli@math.univ-paris13.fr}
\date{\today}

\begin{abstract}
We give a metric characterization of the scalar curvature of a smooth Riemannian manifold, analyzing the maximal distance between $(n+1)$ points in infinitesimally small neighborhoods of a point.
 Since this characterization is purely in terms of the distance function, it could be used to approach the problem of defining the scalar curvature on a non-smooth metric space. In the second part we will discuss this issue, focusing in particular on Alexandrov spaces and surfaces with bounded integral curvature.
\end{abstract}

\maketitle

\tableofcontents

\section{Introduction}

It is well known that the volume growth of the geodesic balls of an $n$-dimensional Riemannian manifold $(M,g)$ is tightly related to curvature. On the one hand, the first non-trivial term of the asymptotic expansion of the volumes of infinitesimally small balls centered at a point $x\in M$ is given, up to a constant, by the scalar curvature at $x$. Namely one has
\begin{equation}\label{vol-exp}
\vol_g (B_\e^M(x))=\vol (B_\e^{\R^n})\left(1-\scal_g(x)\frac{n+2}{6}\e^2 +o(\e^2)\right),
\end{equation}
see for instance \cite{Gray}. On the other hand, thanks to the celebrated Bishop-Gromov volumes comparison theorem, complete manifolds with globally lower bounded Ricci (or sectional) curvature, enjoy an upper control on the volume growth of their geodesic balls, i.e. for all $x\in M$ and $r\geq 0$
\[
\ric_g\geq (n-1)K \Rightarrow \vol_g (B_r^M(x))\leq \vol (B_r^{S_K^n}(x)),
\]
$S_K^n$ being the simply connected $n$-dimensional space forms of constant sectional curvature $K$. Moreover, the equality in Bishop-Gromov is realized if and only if $B_r^M(x)$ and $B_r^{S_K^n}(x)$ are isometric. This latter characterization of the equality case can be seen as an extremal result for the recognition problem. Namely, a recognition problem in Riemannian geometry
``asks for the identification of an unknown riemannian manifold via measurements
of metric invariants on the manifold'', \cite{GM-JAMS}. In their program intended to attack extremal recognition problems, K. Grove and S. Markvorsen introduced a new metric invariant of a Riemannian manifold $X$ (or of a more general metric space), which they called the $q$\textsl{-extent}. This roughly measures how far $q$-points of the space $X$ can be one from the others. Namely, one defines
$$
\mathrm{xt}_q(X):=\left(\begin{array}{c}q \\ 2\end{array}\right)^{-1}\sup_{(x_1,\dots,x_q)\in X^q}\sum_{1\leq i<j\leq q}d_X(x_i,x_j),
$$
where the supremum is obviously a maximum when $X$ is compact. In this case, following \cite{GM-JAMS}, we call $q$\textsl{-extender} the set of $q$ points realizing the $q$-extent. Note that the $q$-extent naturally generalizes the diameter (i.e. the $2$-extent) of a metric space and it is intimately related to the $q$-th packing radius
\[
\operatorname{pack}_q X:= \frac 12 \max_{x_1,\dots x_q}\min_{1\leq i<j\leq q}d_X(x_i,x_j);
\]\cite{GM-JAMS,GM-Bull}. 
One of the purposes of Grove and Markvorsen's recognition program was to study the relation between the $q$-extent and global lower bound on the sectional curvature. A (particular case of) a result of \cite{GM-JAMS} says that an $n$-dimensional, $n\geq 2$, Alexandrov space $X$ with sectional curvature (in the sense of Alexandrov) greater or equal than $1$ satisfies
\begin{equation}\label{eq_GM}
\mathrm{xt}_{n+1}(X) \leq \mathrm{xt}_{n+1}(\mathbb S^n),
\end{equation}
$\mathbb S^n$ being the unit sphere of $\R^{n+1}$ with its canonical metric. Moreover, the equality is realized in  \eqref{eq_GM} if and only if $\operatorname{diam} X=\pi$, which, together with Toponogov's diameter sphere theorem, implies that $X$ is necessarily isometric to $\S^n$ provided it is a Riemannian manifold.\footnote{Further results involving the $q$-extent can be found in \cite{Yang-Duke,JMcES} or the survey \cite{Ma-jugo}.}

According to what said above, it is then natural to expect that, similar to the volume growth, the first nontrivial term in the asymptotic expansion of the $(n+1)$-extent of infinitesimally small geodesic balls centered at a point $x\in M$ involves the scalar curvature at $x$. This is the content of the next result.

\begin{theorem}\label{th_main}
Let $(M,g)$ be a smooth $n$-dimensional Riemannian manifold and $x\in \mathrm{int}(M)$. Then
\begin{align*}
\mathrm{xt}_{n+1}(\bar B_\e^M(x))&=\mathrm{xt}_{n+1}(\bar B_1^{\R^n})\e-\frac{1}{6}\sqrt{\frac{n+1}{2n^5}}\scal_g(x)\e^3+\co(\e^4)\\&=\sqrt{2\frac{n+1}{n}}\e\left\{1-\frac{1}{12n^2}\scal_g(x)\e^2+\co(\e^3)\right\},
\end{align*}
as $\e\to 0$.
\end{theorem}

Since
\[
\mathrm{xt}_{n+1}(\bar B_1^{\R^n})\e=\mathrm{xt}_{n+1}(\bar B_\e^{\R^n})=\sqrt{\frac{2(n+1)}{n}},
\]
see Lemma \ref{lem_lil} below, we trivially deduce

\begin{corollary}\label{coro_main}
Let $(M,g)$ be a smooth $n$-dimensional Riemannian manifold and $x\in \mathrm{int}(M)$. Then
\begin{equation}\label{eq_def-scal}
\mathrm{Scal}_g(x)= \lim_{\e\to 0}\frac{12n^2}{\e^2}\left\{1-
\frac{\mathrm{xt}_{n+1}(\bar B_\e^M(x))}{\mathrm{xt}_{n+1}(\bar B_\e^{\R^n})}\right\}.
\end{equation}
\end{corollary}

Here and on $\bar B_\e^X(x)$ denotes the closed metric ball $\{y\in X\ :\ d_X(x,y)\leq \epsilon\}$ of the metric space $(X,d_X)$. With an abuse of notation, sometime we will write $\bar B_\e^d(x)$ to specify the metric we are considering on a (fixed) space $X$.

The main ingredients in the proof of Theorem \ref{th_main} are
\begin{itemize}
\item An asymptotic formula for the distance function in geodesic normal coordinates, \cite{Br-CQG,Gallot}. On the tangent space at $x$, one can consider two different distances between two vectors $u$ and $v$, that is the Euclidean one $|u-v|$ and the distance induced by the Riemannian metric $g$ pulled-back via the exponential map, i.e. $d_g(\exp_x(u),\exp_x(v))$. Their difference is given (at the first nontrivial order) by $R_g(u,v,v,u)$, where $R_g$ is the Riemann tensor of $g$; see Lemma \ref{lem_br} below.
\item According to \cite{Li}, the only $(n+1)$-extenders in the $n$-dimensional Euclidean ball are the regular simplexes. Here we need a quantitative analysis characterizing the $(n+1)$-simplexes which almost realize the extent; see Proposition \ref{lem_lil-sta} below.
\item By its very definition it turns out that the scalar curvature is (twice) the average of the sectional curvatures of all the planes spanned by an orthonormal frame. The same holds true if instead one considers all the planes spanned by the vertices of a regular $(n+1)$-simplex; see Lemma \ref{lem_normal}.
\end{itemize}

In view of possible applications to metric spaces, some pathological behavior can arise in the $(n+1)$-extenders of geodesic balls. For instance, for thin $2$-dimensional cones, the $3$-extent of a metric ball around the vertex is realized by 3 points including the vertex itself. In some non-rigorous sense the extender's shape is thus discontinuous with respect to the width of the angle. For this reason, we are led to introduce a slightly modified object, which we call 
\textit{boundary $(n+1)$-extent}  and which can be defined for a closed geodesic ball $\bar B^M_\epsilon$ as 
$$
\mathrm{\partial xt}_{n+1}(\bar B^M_\epsilon):=\left(\begin{array}{c}n+1 \\ 2\end{array}\right)^{-1}\sup_{(x_1,\dots,x_q)\in (\partial B_\epsilon)^q}\sum_{1\leq i<j\leq q}d_M(x_i,x_j).
$$
As it is clear from the proof, the Riemannian characterization of the scalar curvature given in Theorem \ref{th_main} and Corollary \ref{coro_main} holds without changes if one replaces in the formulas the $(n+1)$-extent of the geodesic balls with their boundary $(n+1)$-extent. In particular we have the following

\begin{theorem}\label{th_main_bdy}
Let $(M,g)$ be a smooth $n$-dimensional Riemannian manifold and $x\in \mathrm{int}(M)$. Then
\begin{equation}\label{eq_def-scal_bdy}
\mathrm{Scal}_g(x)= \lim_{\e\to 0}\frac{12n^2}{\e^2}\left\{1-
\frac{\mathrm{\partial xt}_{n+1}(\bar B_\e^M(x))}{\partial \mathrm{xt}_{n+1}(\bar B_\e^{\R^n})}\right\}.
\end{equation}
\end{theorem}

Clearly, as in Corollary \ref{coro_main}, also the formula \eqref{vol-exp}, as well as the well-known similar expansion for the area of the surface of small geodeisc balls, allows an alternative definition of scalar curvature at $x\in M$ depending on the geometry of (arbitrarily small) neighborhood of $x$. However, the characterization of the scalar curvature expressed via the $(n+1)$-extent has the peculiarity of depending explicitely on the sole distance function of $M$, and not on the whole Riemannian structure. By a theoretical point of view, this advantage is negligible since a) the distance function uniquely determines the Riemannian metric tensor, \cite{Palais}, and b) whenever the volume measure of the underlying space is given by its $n$-dimensional Hausdorff measure, also the volume measure depends on the distance function. However, the explicit dependence given  in \eqref{eq_def-scal} seems nontrivial. Moreover, this point of view suggests that \eqref{eq_def-scal} could be used to give a definition of scalar curvature (bounds) of a metric space. 

In this direction, we will propose some possible approaches. First, \eqref{eq_def-scal} can be used as it is to introduce a notion of point-wise defined (dimensional) scalar curvature on a metric space $(X,d)$. The asymptotic limit in \eqref{eq_def-scal} in general will not exist, however a $\liminf$ (resp. $\limsup$) version can still be used to define spaces with lower (resp. upper) bounded scalar curvature, see Definition \ref{def_scal} and \ref{def_scalbd} below. This notion of metric scalar curvature reveals consistent at least on metric spaces with lower bounded curvature in the sense of Alexandrov (on which a  unique natural concept of integer dimension is defined). In fact, we have the following result. The inequality part was already observed in \cite[(7)]{GM-JAMS}, while the rigidity is proven in Proposition \ref{prop_rig} below).

\begin{theorem}
Let $(X,d)$ be a $n$-dimensional $CBB(k)$ Alexandrov space, for some $k\in\R$. Then for all $x\in X$ and $\epsilon>0$,
\begin{equation}\label{eq_Alex}
\mathrm{xt}_{n+1}(\bar B_\e^X(x))\leq \mathrm{xt}_{n+1}(\bar B_\e^{S_k^n}),
\end{equation}
where $S_k^n$ is the simply connected $n$-dimensional space form of constant curvature $k$.\\
Moreover, if $\epsilon\leq \frac{\pi}{4\sqrt k}$ or $k\leq 0$, then equality holds in \eqref{eq_Alex} if and only if $\bar B_\e^X(x)$ contains an isometric copy of $\Delta_k^{n+1}(\epsilon)$ with totally geodesic interior. Here, $\Delta_k^{n+1}(\epsilon)$ is the unique (up to isometries) regular $(n+1)$-simplex inscribed in $\bar B_\e^{S_k^n}.$
\end{theorem}
A special class of two dimensional metric spaces is given by the surfaces with bounded integral curvature introduced by Alexandrov. These spaces have a natural notion of curvature measure, which can be described for instance as a weak limit of integral curvatures along an approximating sequence of smooth Riemannian metrics (for details, see the references given in Section \ref{sect_machi}). In Proposition \ref{pr_equal-Machi} we will show that the regular part of this curvature measure can be point-wisely obtained via the local extent as in \eqref{eq_def-scal}. Clearly on the support of the singular part of the curvature measure, the $3$-extent of geodesic $\e$-balls does not converge as $\e\to 0$, so that a point-wise definition is there ill-posed.

Because of the intrinsic singularity of the underlying spaces, it would be more convenient to define on a metric space a scalar curvature measure instead of a point-wise definition. I'm grateful to J. Bertrand and M. Gromov for pointing this out to me. In the interesting paper \cite{KLP}, Kapovitch, Lytchak and Petrunin proposed a notion of \textit{mm-curvature measure} generalizing the asymptotic characterization \eqref{vol-exp}. In particular they proved that such a measure is well-defined on BIC surfaces, although unfortunately it does not coincide with the intrinsic curvature measure. In Section \ref{sect_scalmeas} we will propose a construction similar to that of \cite{KLP}, generalizing to measures on non-smooth surfaces the expansions of the extent and of the boundary extent of geodesic balls (formulas \eqref{def_scal} and \eqref{def_scalbd}) instead of the expansion of the volumes \eqref{vol-exp}.
Namely, consider a smooth $n$-dimensional Riemannian manifold $(M,g)$. For some $r_0>0$ small enough define two families of measure $\{e_r\}_{0<r<r_0}$ and $\{\partial e_r\}_{0<r<r_0}$
by
\[e_r:=\left(1-\frac{\xt(B_r(x))}{\xt(\bar B^{\R^n}_r)}\right)\mathcal H^n\]
and
\[\partial e_r:=\left(1-\frac{\partial\xt(\bar B_r(x))}{\partial\xt(\bar B^{\R^n}_r)}\right)\mathcal H^n.\]
In view of Theorem \ref{th_main} and \ref{th_main_bdy} we have that both $r^{-2}e_r$ and $r^{-2}\partial e_r$ converge weakly in the sense of measure to $\scal_g\cdot \mathcal H^n$. In the non-smooth case, we can prove that on a surface with bounded integral curvature the families $\{r^{-2}e_r\}$ and $\{r^{-2}\partial e_r\}$ are uniformly bounded, so that at least some subsequence converges; see Theorem \ref{th_locfincurv}.
 
The paper is organized as follows. In Section \ref{sect_proof} we prove the Riemannian characterization given by Theorem \ref{th_main} and propose a possible generalization involving the $p$ powers of the distances, see Theorem \ref{th_main_p}. In Section \ref{sect_metric} we consider the point-wise definition of (bounds on the) scalar curvature on Alexandrov spaces. In the last two sections we focus on surfaces with bounded integral curvature. In Section \ref{sect_machi} we show that the point-wise characterization holds on the regular part of BIC surfaces, while in Section \ref{sect_scalmeas} we introduce the (boundary) extent curvature measures inspired by the construction of \cite{KLP}.

\section{Proof of the Riemannian characterization}\label{sect_proof}

Before starting the proof of Theorem \ref{th_main}, let us remark that the result is trivial for $M=\R^n$. In this case, the value of $\mathrm{xt}_{n+1}(\bar B_1^{\R^n})$ is given by the following lemma; see \cite[Lemma 3]{Li}.
\begin{lemma}\label{lem_lil}
It holds
$$\mathrm{xt}_{n+1}(\bar B_1^{\R^n})=\sqrt{2\frac{n+1}{n}}.
$$
Moreover the regular $(n+1)$-simplexes inscribed in $\bar B_1^{\R^n}$ are the only $(n+1)$-tuples of points realizing the equality.
\end{lemma}

We recall here some basic useful facts concerning regular simplexes in the Euclidean space. Fix $\{e_i\}_{i=1}^n$ the canonical orthonormal basis of $\R^n$. For $p=1,\dots,n$ and for $k=1,\dots,p+1$, we introduce the vectors $S^p_k\in \R^n$ defined by 
$S^p_1=e_p$ for $p=1,\dots,n$ and by the inductive rule
\begin{equation}\label{inductiverule}
S_{k}^p=-\frac{1}{p}S_{1}^p+\frac{\sqrt{p^2-1}}{p} S_{k-1}^{p-1}.
\end{equation}
It is easy to see that 
\begin{equation}\label{simplex}
\mathcal S^n=\{S^n_k\}_{k=1}^{n+1}\in \ton{\R^n}^{n+1}
\end{equation}
is a regular $(n+1)$-simplex inscribed in $\partial B_1^{\R^n}$. Similarly,
for any $1\leq p\leq n$, the points  $\{S^p_j\}_{j=1}^{p+1}$ form a regular $(p+1)$-simplex inscribed in $\partial B_1^{\R^{p}}\times\{0_{\R^{n-p}}\}$. For convenience we consider $\{S^1_1,S^1_2\}=\{e_1,-e_1\}$ as a regular 2-simplex of $\R^1\times\{0_{\R^{n-1}}\}$.

In the following, we will need a stable version of the equality case in Lemma \ref{lem_lil}, that is, simplexes which approximate the equality are almost regular. From now on, following standard notation, we will say that a (possibly vector-valued) function $f:(0,\e_0)\to\R$ for some $\e_0>0$ satisfies $f(\e)=\co(\e^a)$ for some $a\in\mathbb Z$ if there exists a constant $C$ (possibly depending on the underlying manifold $M$) such that $\limsup_{\e\to 0}|f(\e)|\e^{-a}\leq C$.

As announced above, we have the following

\begin{proposition}\label{lem_lil-sta}
Let $\e\mapsto P_{\e,k}\in B_1^{\R^n}$, $k=1\dots n+1$, be $n+1$ vector valued functions defined on $(0,\e_0)$ for some small $\e_0>0$. If
\begin{equation}\label{eq_sides}
\summ 1 {n+1} |P_{\e,k}-P_{\e,j}|- \sqrt{\frac{n(n+1)^3}{2}} =\co(\e^2),\quad\text{as }\e\to0,
\end{equation}
then 
\[
|P_{\e,k}-P_{\e,j}| -\sqrt{\frac{2(n+1)}{n}}=\co(\e),\quad\text{as }\e\to0,
\]
for any $1\leq k<j \leq n+1$. Moreover, fixed a regular $(n+1)$-simplex $\mathcal S^n=\{S^n_k\}_{k=1}^{n+1}\in \ton{\R^n}^{n+1}$ inscribed in $\partial B_1^{\R^n}$, there exists a function $A:(0,\e_0)\to O(n)$ (taking values in the isometries group of $\bar B_1^{\R^n}$) such that 
\[
|P_{\e,k}-A(\e)S^n_k|=\co(\e)
\]
as $\e\to 0$ for all $1\leq k\leq n+1$.
\end{proposition}
\begin{proof}
Let $G$ be the barycenter of the given points, i.e. $G=\frac 1{n+1}\sum_{k=1}^{n+1}P_{\e,k}$. Reasoning as in \cite[Lemma 3]{Li}, we get
\begin{align}\label{quant-lil}
&\frac{n(n+1)^3}{2}-\frac{n(n+1)^3}{2}|G|^2- \frac{n(n+1)^2}{2}\sum_{k=1}^{n+1}(1-|P_{\e,k}|^2)\\
&\geq \frac{n(n+1)^2}{2}\sum_{k=1}^{n+1}|P_{\e,k}|^2  -\frac{n(n+1)^3}{2}|G|^2\nonumber\\
&=\frac{n(n+1)^2}{2}\sum_{k=1}^{n+1}|P_{\e,k}-G|^2\nonumber\\
&=\frac{n(n+1)}{2}\summ1{n+1}|P_{\e,k}-P_{\e,j}|^2\nonumber\\
&=\ton{\summ1{n+1}|P_{\e,k}-P_{\e,j}|}^2\frac 1{\cos^2\alpha}\nonumber
\end{align}
where $\alpha$ is the angle at the origin in $\R^{n(n+1)/2}$ formed by the vector $E=(|P_{\e,k}-P_{\e,j}|)_{1\leq k<j\leq n+1}\in\R^{n(n+1)/2}$ and by the vector $I\in\R^{n(n+1)/2}$ all whose components are $1$. Condition \eqref{eq_sides} ensures that $\cos\alpha\neq 0$ for $\e$ small enough and that
\begin{align*}
\frac{n(n+1)^2}{2}\qua{(n+1)\sin^2\alpha +\cos^2\alpha\ton{(n+1)|G|^2+\sum_{k=1}^{n+1}(1-|P_{\e,k}|^2)}}\leq \co(\e^2).
\end{align*}
In particular $\alpha = \co(\e)$, $|G|=\co(\e)$, and 
\begin{equation}\label{lenghtpj}
(1-|P_{\e,k}|^2)=\co(\e^2),\qquad\forall\; k=1,\dots, n+1.
\end{equation}
Let $\nu\in \R$ be the constant such that $\nu I$ is the projection of $E$ onto the line spanned by $I$. Since $\alpha = \co(\e)$, we deduce that $|E-\nu I|\leq |\nu I|\tan\alpha= \co(\e)$, i.e. $|P_{\e,k}-P_{\e,j}|=\nu+\co(\e)$ for all $1\leq k<j\leq n+1$.
Assumption \eqref{eq_sides} implies $\nu=\sqrt{2\frac{n+1}{n}}$.

We introduce vectors $P^p_{\e,k}\in \R^n$ for $p=1,\dots,n$ and $k=1,\dots,p+1$ which are defined as follows. For $p=n$ we set $P^n_{\e,k}:=P_{\e,k}$ for $k=1,\dots,n+1$, while for $p<n$ they are defined inductively on $p$ by the same recursive relation as in  \eqref{inductiverule}, that is, 
\begin{equation}\label{recursive-p}
P_{\e,k-1}^{p-1}:=\frac{p}{\sqrt{p^2-1}}\ton{P_{\e,k}^p+\frac{1}{p}P_{\e,1}^p},\quad k=2,\dots,p+1.
\end{equation}

\begin{lemma}
With notations above, we have that
\begin{align}\label{eq_lem2}
|P_{\e,1}^{1}+P_{\e,2}^{1}|=\co(\e).
\end{align}
Moreover for all $1\leq k<p\leq n$ and for all $1\leq j \leq k+1$ it holds
\begin{align}\label{zz4}
\langle P^{p}_{\e,1},P^k_{\e,j}\rangle = \co(\e),
\end{align}
and
\begin{align}\label{zz5}
\langle P^{p}_{\e,1},P^p_{\e,j}\rangle =\begin{cases} -\frac1p +\co(\e),&\text{if }j>1,\\
1+\co(\e),&\text{if }j=1,\end{cases}
\end{align}
where $\langle\cdot,\cdot\rangle$ is the standard scalar product of $\R^n$.
\end{lemma}

\begin{proof}
The defining relations \eqref{recursive-p} imply
\begin{align*}
\sum_{j=1}^{p}P_{\e,j}^{p-1}
&=\frac{p}{\sqrt{p^2-1}}\sum_{j=2}^{p+1}\ton{
P_{\e,j}^{p}+\frac1p P_{\e,1}^{p}}\\
&=\frac{p}{\sqrt{p^2-1}}\qua{\sum_{j=2}^{p+1}\ton{
P_{\e,j}^{p}}+ P_{\e,1}^{p}}
=\frac{p}{\sqrt{p^2-1}}\sum_{j=1}^{p+1}
P_{\e,j}^{p},
\end{align*}
for any $p=2,\dots,n$. Since $|G|=|\sum_{j=1}^{n+1}
P_{\e,j}^{n}|=\co(\e)$,
applying recursively the latter relation we get \eqref{eq_lem2}.

To prove \eqref{zz4}, first remark that
\begin{align}\label{zz6}
2\langle P^{n}_{\e,1},P^{n}_{\e,j+1}\rangle
&=|P^{n}_{\e,1}|^2+|P^{n}_{\e,j+1}|^2-|P^{n}_{\e,1}-P^{n}_{\e,j+1}|^2\\
&=2(1+\co(\e))-\ton{2\frac{n+1}{n}+\co(\e)}\nonumber\\
&=\frac{2}{n}+\co(\e),\nonumber
\end{align}
which in turn implies
\begin{align}\label{zz}
\langle P^{n}_{\e,1},P^{n-1}_{\e,j}\rangle &= \frac{n}{\sqrt{n^2-1}}\langle P^{n}_{\e,1},P^{n}_{\e,j+1} +\frac 1n P^{n}_{\e,1}\rangle\\\nonumber
&= \frac{|P^n_{\e,1}|^2}{\sqrt{n^2-1}}+ \frac {n\langle P^{n}_{\e,1},P^{n}_{\e,j+1}\rangle}{\sqrt{n^2-1}}= \co(\e).
\end{align}
Moreover, for all $k,q\in\{1,\dots, n\}$ and all $1\leq j\leq k+1$, 
\begin{align}\label{zz2}
\langle P^{q}_{\e,1},P^{k-1}_{\e,j-1}\rangle &= \frac{k}{\sqrt{k^2-1}}\langle P^{q}_{\e,1},P^{k}_{\e,j} +\frac 1k P^{k}_{\e,1}\rangle\\
&=\frac{k}{\sqrt{k^2-1}}\langle P^{q}_{\e,1},P^{k}_{\e,j}\rangle+\frac{1}{\sqrt{k^2-1}}\langle P^{q}_{\e,1},P^{k}_{\e,1}\rangle.\nonumber
\end{align}
Applying this latter recursively with $q=n$ and $k=p+1,\dots,n$ gives, together with \eqref{zz},
\begin{align}\label{zz3}
\langle P^{n}_{\e,1},P^{p}_{\e,j}\rangle = \co(\e) \end{align}
for all $1\leq p\leq n-1$ and all $1\leq j\leq p$. Applying again \eqref{zz2} with $q=p$ and $k\leq n+1$, using also \eqref{zz3}, gives \eqref{zz4}.
%
Finally, we prove \eqref{zz5} once again by induction on $p$. By \eqref{lenghtpj} and \eqref{zz6}, \eqref{zz5} is verified for $p=n$. Now, suppose that \eqref{zz5} is verified for some $p\leq n$. Then
\begin{align*}
\langle P^{p-1}_{\e,j-1},P^{p-1}_{\e,k-1}\rangle
&=\frac{p^2}{p^2-1}\langle P^{p}_{\e,j}+\frac1pP^{p}_{\e,1} , P^{p}_{\e,k}+\frac1pP^{p}_{\e,1}\rangle\\
&=\frac{p^2}{p^2-1}\ton {\langle P^{p}_{\e,j}, P^{p}_{\e,k}\rangle+\frac{1}{p}\ton{\langle P^{p}_{\e,j}, P^{p}_{\e,k}\rangle+\langle \frac1pP^{p}_{\e,1} , \frac1pP^{p}_{\e,1}\rangle}+\frac{1}{p^2}|P^{p}_{\e,1}|^2}\\
&=\frac{p^2}{p^2-1}\ton {\langle P^{p}_{\e,j}, P^{p}_{\e,k}\rangle-\frac{1}{p^2}+\co(\e)},
\end{align*}
so that by the inductive assumption
\begin{align*}
\langle P^{p-1}_{\e,j-1},P^{p-1}_{\e,k-1}\rangle  =\begin{cases} -\frac1{p-1} +\co(\e),&\text{if }j\neq k,\\
1+\co(\e),&\text{if }j=k.\end{cases}
\end{align*}
\end{proof}

An application of the following Lemma with $p=n$ will conclude the proof of Proposition \ref{lem_lil-sta}.

\begin{lemma}
For every $1\leq p\leq n$, there exists an isometry $A_p\in O(n)$ of $B^{\R^n}_1$ satisfying
\begin{align}\label{recursive-A}
P^q_{\e,j}=A_pS^q_{j}+\co(\e),\qquad \forall\;1\leq q\leq p,\ \forall\;1\leq j\leq q+1.
\end{align}
\end{lemma}

\begin{proof}
We proceed by induction on $p$. Let $A_1\in O(n)$ be any isometry of $B^{\R^n}_1$ satisfying $A_1S^n_1=\frac{P^1_{\e,1}}{|P^1_{\e,1}|}$. According to \eqref{lenghtpj}, $P^1_{\e,1}=A_1S^1_1+\co(\e)$. Using \eqref{eq_lem2}, we get
\begin{equation}
P^1_{\e,2}=-P^1_{\e,1}+\co(\e)=-A_1S^1_1+\co(\e)=A_1S^1_2+\co(\e).
\end{equation}
In particular \eqref{recursive-A} is satisfied for $p=1$.

Now, suppose that \eqref{recursive-A} is satisfied for some $1\leq p<n$. Composing $A_p$ with a suitable further isometry $A_p'$, we can find a new isometry $A_{p+1}\in O(n)$ such that \eqref{recursive-A} is satisfied for $p+1$ instead of $p$. Namely, recall that the vectors $A_pS^q_{j}$, with $1\leq q\leq p$ and $1\leq j\leq q+1$ are all contained in a $p$-dimensional hyperplane $\mathcal H_p$of $\R^n$. Moreover $A_pS^{p+1}_{1}\in \mathcal H^\perp_p<\R^n$. Then one can take the isometry $A'_p\in O(n)$ which fix $\mathcal H_p$ (i.e. $A'_p\in I_{\R^p}\otimes O(n-p)$, where $O(n-p)=\mathrm {Iso}(\mathcal H^\perp_p)$) and such the projection of $P^{p+1}_{\e,1}$ onto the $(n-p)$-dimensional hyperplane $\mathcal H^\perp_p<\R^n$ is parallel to $A'_{p}A_pS^{p+1}_1$. Explicitly one has that for each $\e$ there exists $A'_p\in O(n)$ and a unique positive $\alpha$ such that
\begin{align*}
A_p'A_pS^q_{j}=A_pS^q_{j},\qquad \forall\;1\leq q\leq p,\ \forall\;1\leq j\leq q+1,
\end{align*}
and
\begin{align*}
\alpha A_p'A_pS^{p+1}_{1} &= P^{p+1}_{\e,1} - \pi_{\operatorname{span} \{A_pS^j_1\}_{j=1}^p}P^{p+1}_{\e,1}\\
&=P^{p+1}_{\e,1} - \sum_{j=1}^p \langle P^{p+1}_{\e,1},A_pS^j_1\rangle A_pS^j_1.
\end{align*}
According to \eqref{recursive-A}, using also \eqref{zz4} and \eqref{zz5}, we deduce
\begin{align*}
\abs{\alpha A_p'A_pS^{p+1}_{1} - P^{p+1}_{\e,1}}
&=\abs{\sum_{j=1}^p \langle P^{p+1}_{\e,1},P^{j}_{\e,1}+\co(\e)\rangle (P^{j}_{\e,1}+\co(\e))}\\
&=\abs{\sum_{j=1}^p \langle P^{p+1}_{\e,1},P^{j}_{\e,1}\rangle P^{j}_{\e,1}+\co(\e)}\\
&=\co(\e).
\end{align*}
Again by \eqref{zz5} we have also that $\alpha=\alpha(\e)=1+\co(\e)$. Setting $A_{p+1}:=A_p'A_p$, we have proved \eqref{recursive-A} for $p+1$, hence recursively for every $1\leq p\leq n$.
\end{proof}
\end{proof}

\bigskip
We now come back to the proof of the main theorem.

\begin{proof}[Proof (of Theorem \ref{th_main})]
Fix a point $x_0\in M$ and let $\epsilon_0$ small enough so that for all $0<\e<\e_0$, $\phi_\e=\exp_{x_0}(\e\cdot):\bar B_1^{\R^n}(0)\to \bar B_\e^M(x_0)$ is a diffeomorphism.

For every $\e\in (0,\e_0)$, let $\{Q_{\e,j}\}_{j=1}^{n+1}$ be an $(n+1)$-extender of $\bar B_\e^M(x_0)$, i.e.
$$
\frac{n(n+1)}{2}\xt(\bar B_\e^M(x_0))=\sum_{1\leq k<j\leq n+1} d_g(Q_{\e,k},Q_{\e,j}).$$ 
For $k=1,\dots,n$, define $P_{\e,k}:=\phi_\e^{-1}(Q_{\e,k})$.

We need the following

\begin{lemma}\label{lem_br}
Fix a point $x_0\in M$ and let $\epsilon_0$ small enough so that for all $0<\e<\e_0$, $\phi_\e=\exp_{x_0}(\e\cdot):\bar B_1^{\R^n}(0)\to \bar B_\e^M(x_0)$ is a diffeomorphism. Then for every distinct points $p_1,p_2\in \bar B_1^{\R^n}(0)$ we have
\begin{align*}
d_g(\phi_\e(p_1),\phi_\e(p_2))&=\e|p_1-p_2|-\frac{\e^3}{6|p_1-p_2|}R_{x_0}(p_1,p_2,p_2,p_1)+\co(\e^4),
\end{align*}
where $R_{x_0}$ is the Riemann tensor of $(M,g)$ at $x_0$.
Here and on, $\R^n$ is canonically identified with $T_{x_0}M$. \end{lemma}

\begin{proof}
This follows easily from the asymptotic formula
\begin{equation}
d_g^2(\phi_\e(p_1),\phi_\e(p_2))=\e^2|p_1-p_2|^2-\frac{\e^4}{3}R_{x_0}(p_1,p_2-p_1,p_2-p_1,p_1)+\co(\e^5),
\end{equation}
which can be proved via explicit computations; see for instance \cite{Gallot,Br-CQG}.
\end{proof}

Lemma \ref{lem_br} gives that

\begin{align}\label{xt_form}
\frac{n(n+1)}{2}\xt(\bar B_\e^M(x_0))&=\e\sum_{1\leq k<j\leq n+1}\left|P_{\e,k}-P_{\e,j}\right| \\
&-  \frac{\e^3}{6}\sum_{1\leq k<j\leq n+1}\frac {R_{x_0}(P_{\e,k},P_{\e,j},P_{\e,j},P_{\e,k})}{|P_{\e,k}-P_{\e,j}|}+\co(\e^4).\nonumber
\end{align}
Note that 
\begin{align*}
\frac {R_{x_0}(P_{\e,k},P_{\e,j},P_{\e,j},P_{\e,k})}{|P_{\e,k}-P_{\e,j}|}
&=\frac {R_{x_0}(P_{\e,k}-P_{\e,j},P_{\e,j},P_{\e,j},P_{\e,k}-P_{\e,j})}{|P_{\e,k}-P_{\e,j}|}\\
&=\frac{Sect_{x_0}(P_{\e,k}\wedge P_{\e,j})\qua{|P_{\e,k}-P_{\e,j}|^2|P_{\e,j}|^2-\left\langle P_{\e,j},P_{\e,k}-P_{\e,j}\right\rangle^2}}{|P_{\e,k}-P_{\e,j}|}
\\
&\leq2 |Sect_{x_0}(P_{\e,k}\wedge P_{\e,j})|,
\end{align*}
since $P_{\e,k}\in \bar B_1^{\R^n}$ for all $k=1,\dots,n$. We get in particular that
\begin{align}\label{xt_form2}
\left| \xt(\bar B_\e^M(x_0))-\e\frac{2}{n(n+1)}\sum_{1\leq k<j\leq n+1}\left|P_{\e,k}-P_{\e,j}\right|\right| \leq
\frac{1}{3}\left\|Sect_{x_0}\right\|\e^3+\co(\e^4).
\end{align}

Now, let $\mathcal S^n$ be the regular $(n+1)$-simplex of $\R^n$ introduced in \eqref{simplex}. Consider $n+1$ points $\{T^n_{\e,k}\}_{k=1}^{n+1}$ in $\bar B_\e^M(x_0)$ defined by $T^n_{\e,k}=\phi_\e(S^n_k)$. 
Again by Lemma \ref{lem_br}, for every $1\leq k<j\leq n+1$
\begin{align}\label{eq_br}
d_g(T^n_{\e,k},T^n_{\e,j})&=\e|S^n_k-S^n_j|-\frac{\e^3}{6|S^n_k-S^n_j|}R_{x_0}(S^n_k,S^n_j,S^n_j,S^n_k)+\co(\e^4)\\
&=\e\sqrt{2\frac{n+1}{n}}-\frac{\e^3}{6}\sqrt{\frac{n}{2(n+1)}}R_{x_0}(S^n_k,S^n_j,S^n_j,S^n_k)+\co(\e^4).\nonumber
\end{align}
This latter, together with \eqref{xt_form2} and Lemma \ref{lem_lil}, yield
\begin{align}
\sum_{1\leq k<j\leq n+1}\left|S^n_k-S^n_j\right|
&\leq \e^{-1}\sum_{1\leq k<j\leq n+1} d_g(T^n_{\e,k},T^n_{\e,j})+\co(\e^2)\\\nonumber
&\leq \e^{-1}\frac{n(n+1)}{2}\xt(\bar B_\e^M(x_0)) +\co(\e^2)\\\nonumber
&\leq \sum_{1\leq k<j\leq n+1}\left|P_{\e,k}-P_{\e,j}\right|+\co(\e^2)\\\nonumber
&\leq \frac{n(n+1)}{2}\xt(\bar B_1^{\R^n})+\co(\e^2)\\\nonumber
&=\sum_{1\leq k<j\leq n+1}\left|S^n_k-S^n_j\right|+\co(\e^2).
\end{align}
In particular, 
$$
\sum_{1\leq k<j\leq n+1}\left|P_{\e,k}-P_{\e,j}\right|= \sqrt{\frac{n(n+1)^3}{2}}+\co(\e^2).
$$
%
%
%
%
%
We can thus apply Proposition \ref{lem_lil-sta} to deduce
\begin{align}\label{approx-P}
&P_{\e,k}=A(\e)S^n_k+\co(\e),\quad\forall 1\leq k\leq n+1.
\end{align}
for some function $A:(0,\e_0)\to O(n)$ taking values in the isometries group of $\bar B_1^{\R^n}$.


We are going to use the following Lemma, which will be proved later.

\begin{lemma}\label{lem_normal}
Let $\mathcal S^n=\{S^n_k\}_{k=1}^{n+1}\in \ton{\R^n}^{n+1}$ be any regular $(n+1)$-simplex inscribed in $\bar B_1^{\R^n}(0)$. Then
\begin{equation} \label{eq_average}
\sum_{1\leq k<j\leq n+1}R_{x_0}(S^n_k,S^n_j,S^n_j,S^n_k)=\frac{(n+1)^2}{2n^2}\mathrm{Scal}_g(x_0).
\end{equation}
\end{lemma}

%
Since for $1\leq k\leq n+1$, $P_{\e,k}\in \bar B_1^{\R^n}(0)$, Lemma \ref{lem_lil} gives that 
\[
\sum_{1\leq k<j\leq n+1}\left|P_{\e,k}-P_{\e,j}\right|\leq \frac{n(n+1)}{2}\xt(\bar B_1^{\R^n}(x_0))= \sum_{1\leq k<j\leq n+1}\left|S^n_{k}-S^n_{j}\right|.
\]
Moreover, according to \eqref{approx-P} and Lemma \ref{lem_normal},
\begin{align*}
&\sum_{1\leq k<j\leq n+1}\frac {R_{x_0}(P_{\e,k},P_{\e,j},P_{\e,j},P_{\e,k})}{|P_{\e,k}-P_{\e,j}|}\\
&= \sum_{1\leq k<j\leq n+1}\frac {R_{x_0}(A(\e)S^n_{k}+\co(\e),
A(\e)S^n_{j}+\co(\e),A(\e)S^n_{j}+\co(\e),A(\e)S^n_{k}+\co(\e))}
{|A(\e)S^n_{k}-A(\e)S^n_{j}|+\co(\e)}\\
&=\sum_{1\leq k<j\leq n+1}\frac {R_{x_0}(A(\e)S^n_{k},
A(\e)S^n_{j},A(\e)S^n_{j},A(\e)S^n_{k})}
{|A(\e)S^n_{k}-A(\e)S^n_{j}|}+\co(\e)\\
&=\sqrt{\frac{n}{2(n+1)}}\sum_{1\leq k<j\leq n+1}R_{x_0}(A(\e)S^n_{k},
A(\e)S^n_{j},A(\e)S^n_{j},A(\e)S^n_{k})+\co(\e)\\
&=\sqrt{\frac{n}{2(n+1)}}\frac{(n+1)^2}{2n^2}\scal_g(x_0)\end{align*}
Inserting these latter in \eqref{xt_form} we get 
\begin{align}\label{eq_fin1}
\frac{n(n+1)}{2}\xt(\bar B_\e^M(x_0))
&\leq \e\sum_{1\leq k<j\leq n+1}\left|S^n_{k}-S^n_{j}\right| -\frac {\e^3}{12}\sqrt{\frac{(n+1)^3}{2n^3}}\scal_g(x_0)+\co(\e^4)\end{align}
On the other hand, from \eqref{eq_br} and Lemma \ref{lem_normal}, we have also
\begin{align}\label{eq_fin2}
\frac{n(n+1)}{2}\xt(\bar B_\e^M(x_0))
&\geq \sum_{1\leq k<j\leq n+1} d_g(T^n_{\e,k},T^n_{\e,k})\\&=
\e \sum_{1\leq k<j\leq n+1}\left|S^n_{k}-S^n_{j}\right| -\frac {\e^3}{12}\sqrt{\frac{(n+1)^3}{2n^3}}\scal_g(x_0)+\co(\e^4).\nonumber\end{align}
The asymptotic estimates \eqref{eq_fin1} and \eqref{eq_fin2}, together with Lemma \ref{lem_lil}, conclude the proof of Theorem \ref{th_main}.
\end{proof}

It remains to prove Lemma \ref{lem_normal}.

\begin{proof}[Proof (of Lemma \ref{lem_normal})]
For $p=1,\dots,n$ and $k=1,\dots,p+1$, define unitary vectors $S^p_k\in\R^n$ as in  \eqref{inductiverule}, and fix an orthonormal basis of $T_{x_0}M=\R^n$ by setting $e_j=S^j_1$ for $1\leq j\leq n$. 

We are going to prove by induction on $n$ that
$$\sum_{1\leq k<j\leq n+1}R_{x_0}(S^n_k,S^n_j,S^n_j,S^n_k)=\frac{(n+1)^2}{n^2}\summ 1{n}Sect(x_0)(e_k\wedge e_j)
.$$
For the shortness of notation, for vectors $X,Y\in\R^n$, we set
$$
\kappa(X,Y):=R_{x_0}(X,Y,Y,X).
$$
For $n=2$, for every $1\leq k<j\leq 3$, $$\kappa(S^2_k,S^2_j)= Sect_{x_0} \left(|S^2_k|^2|S^2_j|^2-\langle S^2_k,S^2_j\rangle^2\right)= (1 -\frac{1}{4})Sect_{x_0} = \frac{3}{4}Sect_{x_0},$$ 
so that 
\begin{equation}\label{initialisation}
\sum_{1\leq k<j\leq 3}\kappa(S^2_k,S^2_j)=\frac{9}{4}Sect_{x_0}=\frac{9}{4}Sect_{x_0}(e_1,e_2).\end{equation}

\begin{lemma}\label{1stfact}
For every $1\leq p \leq n-1$, one has
\begin{equation}\label{1st-ind}
\sum_{k=1}^{p+1}\ka n1pk=\frac{p+1}{p}\sum_{j=1}^p Sect_{x_0}(e_n\wedge e_j).
\end{equation}
\end{lemma}
\begin{proof}
We prove this by induction on $p$. First, note that $\sum_{k=1}^pS_k^{p-1}=0$. Accordingly,
$$\sum_{k=2}^{p+1}R_{x_0}(S_1^n,S_1^p,S_{k-1}^{p-1},S_1^n)=0,$$ and for every $1\leq p<n$,
\begin{align*}
\sum_{k=1}^{p+1}\ka n1pk 
&=\ka n1p{1} + \sum_{k=2}^{p+1} \kappa(S_1^n,-\frac 1p S_1^p+ \frac{\sqrt{p^2-1}}{p} S^{p-1}_{k-1})\\
&= \ka n1p1+\frac p{p^2}\ka n1p1+ \frac{p^2-1}{p^2}\sum_{k=2}^{p+1}\ka n1{p-1}{k-1}\\
&=\left(\frac {p+1}p\right)\ka n1p1 + \frac{p^2-1}{p^2}\sum_{k=1}^{p}\ka n1{p-1}{k}\\
&=\left(\frac {p+1}p\right) Sect_{x_0}(e_n \wedge e_p) + \frac{p^2-1}{p^2}\sum_{k=1}^{p}\ka n1{p-1}{k}.
\end{align*}
Since 
$$
\sum_{k=1}^2\ka n11k = 2 Sect_{x_0}(S^n_1 \wedge S^1_1)=2 Sect_{x_0}(e_n \wedge e_1),$$
Lemma \ref{1stfact} is proved.
\end{proof}

We come back to the proof of Lemma \ref{lem_normal}. Observe that, since $\sum_{j=2}^{n+1}S_{j-1}^{n-1}=0$ and by the linearity of the Riemann tensor, one get that
\begin{align}\label{fact21}
&2\sum_{2\leq k<j\leq n+1} \ri n1{n-1}{k-1}{n-1}{j-1}n1\\
&=\sum_{2\leq k<j\leq n+1}\ton{ \ri n1{n-1}{k-1}{n-1}{j-1}n1+\ri n1{n-1}{j-1}{n-1}{k-1}n1}\nonumber\\
&=\sum_{k=2}^{n+1}\sum_{\substack{j=2\\j\neq k}}^{n+1}\ri n1{n-1}{k-1}{n-1}{j-1}n1 
= - \sum_{k=2}^{n+1}\ri n1{n-1}{k-1}{n-1}{k-1}n1,\nonumber
\end{align}
and
\begin{align}\label{fact22}
&\sum_{2\leq k<j\leq n+1}\ton{ \ri n1{n-1}{k-1}{n-1}{k-1}{n-1}{j-1}+\ri n1{n-1}{j-1}{n-1}{j-1}{n-1}{k-1}}\\\nonumber
&=\sum_{j=2}^{n+1}\sum_{\substack{k=2\\k\neq j}}^{n+1}\ri n1{n-1}{j-1}{n-1}{j-1}{n-1}{k-1}\\\nonumber
&=-\sum_{j=2}^{n+1}\ri n1{n-1}{j-1}{n-1}{j-1}{n-1}{j-1}=0.
\end{align}
Hence, we can compute
\begin{align*}
&\summ 1{n+1}\ka nknj\\\nonumber
&=\sum_{j=2}^{n+1}\ka n1nj + \summ 2{n+1}\ka nknj\\\nonumber
&=\sum_{j=2}^{n+1} \kappa(S_1^n,-\frac1n S_1^n + \frac{\sqrt{n^2-1}}{n} S^{n-1}_{j-1})\\\nonumber
&+ \summ 2{n+1}\kappa(-\frac1n S_1^n + \frac{\sqrt{n^2-1}}{n} S^{n-1}_{k-1},-\frac1n S_1^n + \frac{\sqrt{n^2-1}}{n} S^{n-1}_{j-1})\\\nonumber
&= \frac{n^2-1}{n^2}\sum_{j=2}^{n+1}\ka n1{n-1}{j-1}-2\frac{n^2-1}{n^4}\summ 2{n+1}\ri n1{n-1}{k-1}{n-1}{j-1}n1 \\\nonumber
&+ \frac{n^2-1}{n^4} \summ 2{n+1}\qua{\ka n1{n-1}{j-1}+ \ka n1{n-1}{k-1}}\\
&-\frac{(n^2-1)^{3/2}}{n^4}\summ 2{n+1}\qua{2\ri n1{n-1}{j-1}{n-1}{j-1}{n-1}{k-1}+2\ri n1{n-1}{k-1}{n-1}{k-1}{n-1}{j-1}}\\\nonumber
&+\frac{(n^2-1)^{2}}{n^4} \summ2{n+1} \ka{n-1}{k-1}{n-1}{j-1}.
\end{align*}
Inserting \eqref{fact21} and \eqref{fact22}, and applying Lemma \ref{1stfact} with $p=n-1$, we get 
\begin{align*}
&\summ 1{n+1}\ka nknj\\
&= \frac{(n+1)^2}{n^2}\sum_{j=2}^{n+1}\ka n1{n-1}{j-1}+\frac{(n^2-1)^{2}}{n^4} \summ2{n+1} \ka{n-1}{k-1}{n-1}{j-1}\\
&= \frac{(n+1)^2}{n^2}\qua{\sum_{j=1}^{n-1}Sect_{x_0}(e_n\wedge e_j) + \frac{(n-1)^2}{n^2} \summ 2{n+1}\ka {n-1}{k-1}{n-1}{j-1}}
\end{align*}
This latter, together with \eqref{initialisation}, proves Lemma \ref{lem_normal}.
\end{proof}

\subsection*{A possible generalization}

For real positive $p$, one could also define the $p^{th}$ order $q$-extent of a metric space $(X,d)$ as
$$
\mathrm{xt}^{(p)}_q(X):=\left(\begin{array}{c}q \\ 2\end{array}\right)^{-1}\sup_{(x_1,\dots,x_q)\in X^q}\sum_{1\leq i<j\leq q}(d_X(x_i,x_j))^p.
$$
It turns out that for $1\leq p< 2$, Theorem \ref{th_main} and Corollary \ref{coro_main} generalize to $\mathrm{xt}^{(p)}_q(X)$. Namely one has
\begin{theorem}\label{th_main_p}
Let $(M,g)$ be a smooth $n$-dimensional Riemannian manifold and $x\in \mathrm{int}(M)$. Then
$$
\mathrm{xt}^{(p)}_{n+1}(\bar B_\e^M(x))=\e^p\mathrm{xt}^{(p)}_{n+1}(\bar B_1^{\R^n})-\frac{p}{12 n^2}\left(\frac{2(n+1)}{n}\right)^{p/2}\scal_g(x)\e^{p+2}+\co(\e^{p+3}).
$$
\end{theorem}
\begin{corollary}\label{coro_main_p}
Let $(M,g)$ be a smooth $n$-dimensional Riemannian manifold and $x\in \mathrm{int}(M)$. Then
\begin{equation}\label{eq_def-scal_p}
\mathrm{Scal}_g(x)= \lim_{\e\to 0}\frac{12 n^2}{p\e^{2}}
\left(1-\frac{\mathrm{xt}^{(p)}_{n+1}(\bar B_\e^M(x))}{\mathrm{xt}^{(p)}_{n+1}(\bar B_\e^{\R^n})}\right).
\end{equation}

\end{corollary}

These results can be proved essentially as the case $p=1$ treated above. The main difference is the proof of the stability result stated as  Proposition \ref{lem_lil-sta}. Namely, one has to replace relation \eqref{quant-lil} with the inequality
\begin{align*}
&\summ1{n+1}|P_{\e,k}-P_{\e,j}|^p\\
&\leq\ton{\frac{n(n+1)}{2}}^{1-p/2}\qua{(n+1)^2-(n+1)^2|G|^2-(n+1)\sum_{k=1}^{n+1}\ton{1-|P_{\e,k}|^2}}\mathcal E,\end{align*}
where
$$\mathcal E:= \qua{1-\ton{1-\frac{p}{2}}\summ 1{n+1}\left|\frac{|P_{\e,k}-P_{\e,j}|}{\ton{\summ1{n+1}|P_{\e,k}-P_{\e,j}|^2}^{1/2}}-\sqrt{\frac{2}{n(n+1)}}\right|^2}.
$$
This latter can be obtained using a quantitative version of H\"older inequality; see for instance \cite[Theorem 2.2]{Ald}.

Note that, for $p\geq 2$, regular $(n+1)$-simplexes are no more the (unique) $(n+1)$-extenders of Euclidean balls, so that the proof fails to work in this case.

\section{Local extent in metric spaces}\label{sect_metric}

The curvature of metric spaces is the object of an active research field. Since the seminal work by Alexandrov, a notion of metric sectional curvature lower bounded is provided. For later purposes, we recall that a locally compact metric space $(X,d)$, whose metric $d$ is intrinsic, is a \textsl{space of curvature bounded above} (resp.
\textsl{below}) by $k$ in the sense of Alexandrov if in some neighborhood of each point the following
holds:\\
For every $\Delta abc$ and every point $d\in [ac]$, one has $|db| \leq |\bar d\bar b|$ (resp.$|db| \geq
|\bar d\bar b|$) where $\bar d$ is the point on the side $[\bar a\bar c]$ of a comparison triangle $\Delta \bar a\bar b\bar c$ such
that $|ad| = |\bar a\bar d|$.
Here a comparision triangle is a triangle of vertices $\bar a$, $\bar b$ and $\bar c$ in the $k$-plane $S_k^2$ (i.e. the $2$-dimensional simply connected space of constant curvature $k$) satisfying $|ab| = |\bar a\bar b|$, $|ac| = |\bar a\bar c|$ and $|bc| = |\bar b\bar c|$.
Other equivalent definitions can be found for instance in \cite{BBI}.

In the last decades several possible notions of Ricci curvature bounds for (measured) metric spaced have been proposed and successfully investigated. 

On the other hand, the aim for a metric notion of scalar curvature is much more recent and seems more difficultous; see for instance \cite{Gro-CEJM,Sormani,Gro-arXiv} and references therein. Specific answers to this problem have been given in particular context; see for instance \cite{Machi,KLP} for Alexandrov surfaces and \cite{Ber-AG1,Ber-AG2} for Alexandrov definable sets in $o$-minimal structures.

Here we propose the following definition

\begin{definition}\label{def_scal}
Let $(X,d)$ be a locally compact metric space and $x\in \mathrm{int}(X)$. We define the ($n$-dimensional) scalar curvature of $X$ at $x$ as
\begin{align*}
\mathrm{Scal}^{(n)}_d(x)=
\lim_{\e\to 0}\frac{12n^2}{\e^2}\left\{1-
\frac{\mathrm{xt}_{n+1}(\bar B_\e^X(x))}{\mathrm{xt}_{n+1}(\bar B_\e^{\R^n}(x))}\right\},
\end{align*}
whenever the limit exists.
\end{definition}

Note that, in general, the scalar curvature of a metric space is not point-wise defined (think for instance to the singular part of a polyhedral space). Accordingly, it could be interesting to consider bounds on the scalar curvaure.
\begin{definition}\label{def_scalbd}
Let $(X,d)$ be a locally compact metric space. We say that $X$ has ($n$-dimensional) scalar curvature greater or equal than $k\in\R$ at $x\in X$, and we write $\mathrm{Scal}^{(n)}_d(x)\geq k$, if
$$\liminf_{\e\to 0}\frac{12n^2}{\e^2}\left\{1-
\frac{\mathrm{xt}_{n+1}(\bar B_\e^X(x))}{\mathrm{xt}_{n+1}(\bar B_\e^{\R^n}(x))}\right\}\geq k.
$$
Similarly, we say that $X$ has ($n$-dimensional) scalar curvature smaller or equal than $k\in\R$ at $x\in X$, and we write $\mathrm{Scal}^{(n)}_d(x)\leq k$, if
$$\limsup_{\e\to 0}\frac{12n^2}{\e^2}\left\{1-
\frac{\mathrm{xt}_{n+1}(\bar B_\e^X(x))}{\mathrm{xt}_{n+1}(\bar B_\e^{\R^n}(x))}\right\}\leq k.
$$

\end{definition}
According to Lemma \ref{lem_lil} below, in case $k=0$ the definition above specifies as follows. A metric space $X$ has nonnegative scalar curvature at $x$ if 
$$\limsup_{\e\to 0}\e^{-1}\mathrm{xt}_{n+1}(\bar B_\e^X(x))\leq \sqrt{2\frac{n+1}{n}},
$$
while $X$ has nonpositive scalar curvature at $x$ if 
$$\liminf_{\e\to 0}\e^{-1}\mathrm{xt}_{n+1}(\bar B_\e^X(x))\geq \sqrt{2\frac{n+1}{n}}.
$$

Thanks to Theorem \ref{th_main} and the subsequent discussion, the above definitions of ($n$-dimensional) scalar curvature are clearly consistent with the classical one when the underlying space is a Riemannian manifold. Moreover, they are also consistent with the curvature bounds in the sense of Alexandrov.

Recall that (finite dimensional) Alexandrov spaces with a lower curvature bound have a natural notion of dimension (in fact, all the reasonable notions of dimension coincide, and the dimension is a positive integer; see \cite[Chapter 10]{BBI}). Accordingly, as in the Riemannian setting one expects the 
following result, which is in fact a direct consequence of \cite[Proposition 10.6.10]{BBI}.
\begin{theorem}\label{th_alex}
Let $(X,d)$ be a $n$-dimensional $CBB(k)$ Alexandrov space, for some $k\in\R$, then for all $x\in X$ and $\epsilon>0$,
\begin{equation}\label{comp_xt}
\mathrm{xt}_{n+1}(\bar B_\e^X(x))\leq \mathrm{xt}_{n+1}(\bar B_\e^{S_k^n}),
\end{equation}
where $S_k^n$ is the simply connected $n$-dimensional space form of constant curvature $k$. In particular
$$\scaln_d(x)\geq kn(n-1).$$
\end{theorem}
Let $\Delta_k^{n+1}(\epsilon)$ be the unique (up to isometries) regular $(n+1)$-simplex  inscribed in $\bar B_\e^{S_k^n}$ (that is, with totally geodesic faces in $S_k^n$). Note that $\Delta_0^{n+1}(1)=\mathcal S^{n+1}$.
Similarly to the analogous rigidity result for the packing radius (see \cite[Lemma 3.3]{GM-JAMS}), we have the following
\begin{proposition}\label{prop_rig}
Let $(X,d)$ be a $n$-dimensional $CBB(k)$ Alexandrov space, for some $k\in\R$. Suppose that for some $x\in X$ and $\epsilon>0$ (with $\epsilon \leq \frac \pi{4\sqrt k}$ if $k>0$),
\begin{equation}\label{eq_xt}
\mathrm{xt}_{n+1}(\bar B_\e^X(x))= \mathrm{xt}_{n+1}(\bar B_\e^{S_k^n}).
\end{equation}
Then an isometric copy of $\Delta_k^{n+1}(\epsilon)$ with totally geodesic interior is inscribed in $\bar B_\e^X(x)$. In particular, there exists $\delta_0$ depending on $\epsilon$, $k$ and $n$ such that $\bar B_\delta^X(x))$ is isometric to $\bar B_\delta^{S_k^n}$ for every $0<\delta\leq \delta_0$.
\end{proposition}
\begin{remark}{\rm
As for \cite[Lemma 3.3]{GM-JAMS}, Proposition \ref{prop_rig} is sharp, in the sense that in general we can not expect $\bar B_\e^X(x)$ and $\bar B_\e^{S_k^n}$ to be isometric. A trivial counterexample is given by the Alexandrov space $\Delta_k^{n+1}(\epsilon)$, endowed with the metric induced by $S_k^{n}$.

The bound on $\epsilon$ is given by the fact that regular simplexes are no more extenders of $\bar B_\e^{S_k^n}$ for $\epsilon$ close to $\frac {\pi}{2\sqrt k}$ (although they continue being packers).} 
\end{remark}

\begin{proof}
Let $\{Q_j\}_{j=1}^{n+1}$ be a $(n+1)$-extender for $\bar B_\e^X(x)$. Let $f:X\to S^{n}_k$ be the noncontracting map given by \cite[Proposition 10.6.10]{BBI}, and define $\{P_j\}_{j=1}^{n+1}\subset \bar B_\e^{S_k^n}(x)$ as $P_j=f(Q_j)$.
By the non-contractivity of $f$, it holds 
\begin{equation}\label{comp_length}
d_{S_k^n}(P_j,P_k)\geq d(Q_j,Q_k)
\end{equation} 
for every $k$ and $j$. In particular 
\begin{align*}
\left(\begin{array}{c}n+1\\2
\end{array}\right)\mathrm{xt}_{n+1}(\bar B_\e^X(x))
&=\sum_{1\leq j<k\leq n+1}d(Q_j,Q_k)\\
&\leq \sum_{1\leq j<k\leq n+1}d_{S_k^n}(P_j,P_k)\\
&\leq \left(\begin{array}{c}n+1\\2
\end{array}\right) \mathrm{xt}_{n+1}(\bar B_\e^{S_k^n}).
\end{align*}
Equality in \eqref{eq_xt}, together with \eqref{comp_length}, implies that
\begin{equation}\label{equal_P}
d_{S_k^n}(P_j,P_k)=d(Q_j,Q_k)
\end{equation}
for every $k$ and $j$ and that $\{P_j\}_{j=1}^{n+1}$ is a $(n+1)$-extender for $\bar B_\e^{S_k^n}$, which is hence given by $\Delta_k^{n+1}(\epsilon)$; see \cite[page 9]{GM-JAMS}.  

According to the characterization given in \cite[page 19]{GM-JAMS}, $\{P_j\}_{j=1}^{n+1}$ is also a $(n+1)$-packer of $\bar B_\e^{S_k^n}$. From \eqref{equal_P}, we deduce that $\mathrm{pack}_{n+1}(\bar B_\e^{X})\geq \mathrm{pack}_{n+1}(\bar B_\e^{S_k^n})$. By angle comparison, this inequality is in fact an equality, as observed also in \cite[(3.1)]{GM-JAMS}. By the proof of \cite[Lemma 2.3]{GM-JAMS}, we get that $\bar B_{\e}^{X}$ contains an isometric copy of $\Delta_k^{n+1}(\epsilon)$ with totally geodesic interior and vertexes given by $\{Q_j\}_{j=1}^{n+1}$.
%
%
\end{proof}

In the assumption of an upper bound on the curvature, there is not a general natural notion of dimension. However, one can still introduce a geometric dimension $\textrm{GeomDim}(X)$ of a CBA space $X$, which is defined as the smallest function defined on CBA spaces such that a) $\textrm{GeomDim}(X)=0$ whenever $X$ is a discrete space, and b) $\textrm{GeomDim}(X)\geq 1+ \textrm{GeomDim}(\Sigma_p X)$ for every $p\in X$, $\Sigma_p X$ being the space of directions;  see \cite{Kl-MathZ}. B. Kleiner proved in particular that $\textrm{GeomDim}(X)$ is equal to the topological dimension of the space. Moreover, 
whenever $\textrm{GeomDim}(X)<+\infty$, it holds that
$
\textrm{GeomDim}(X)$ is the greater integer $q$ such that there is an isometric embedding of the standard unit sphere $\S^{n-1}\subset\R^n$ into $\Sigma_pX$ for some $p\in X$. By the  definition of the space of directions and by the monotonicity condition for upper curvature bounds on metric space (see \cite[Section 4.3.1]{BBI}), one easily gets
\begin{proposition}
Let $(X,d)$ be a $CBA(k)$ Alexandrov space. Suppose that there is an isometric embedding of the standard unit sphere $\S^{n-1}\subset\R^n$ into $\Sigma_xX$ for some $x\in X$. Then 
$$\scaln_d(x)\leq kn(n-1).$$
\end{proposition}

\section{Curvature measure of $CBB(k)$ surfaces}\label{sect_machi}

In this section we focus on surfaces with lower bounded curvature. Let $(S,d)$ be a $2$-dimensional topological surface, whose metric is $CBB(k)$ for some $k\in\R$. $(S,d)$ is in particular a surface of bounded integral curvature (BIC) so that it supports a well-defined curvature measure $\omega$, see for instance \cite{Re}, \cite{AZ} or the survey \cite{Tr}. This latter measure constitutes a generalization of the Riemannian curvature of a surface to the non-smooth setting, i.e. $d\omega=\sect_g dA_g=\frac{1}{2}\scal_g dA_g$ whenever the distance $d$ of $(S,d)$ is induced by a smooth Riemannian metric $g$ on $S$. In fact $\omega$ can be defined as the weak limit (in the sense of measures) of the sectional curvature of an approximating sequence of smooth metrics. In \cite{Machi}, Y. Machigashira studied the regular part $\omega_{reg}$ of the curvature measure of a $CBB(k)$ surface with respect to the $2$-dimensional Hausdorff measure (restricted to Borel sets). He introduced the (almost everywhere defined) Gaussian curvature function $G$ on $S$ by
\begin{equation}\label{eq_gauss}
G(x)=\inf_{d>0} \underline{G_d}(x);\qquad\underline{G_d}(x):=\liminf_{\Delta\to \{x\}, x\in\Delta,\Delta\in U_d}\frac{e(\Delta)}{Area(\Delta)}
\end{equation}
and proved that $\omega_{reg}(E)=\int_E G(x)d\mathcal H^2$ for every suitably measurable set $E\subset S$ (see \cite{Machi} for details). The liminf in \eqref{eq_gauss} is taken for geodesic triangles converging to the point $x$, containing $x$ in their interior, and with interior angles greater than $d>0$. Moreover the excess $e(\Delta)$ of the geodesic triangle $\Delta$ is defined so that $e(\Delta)+\pi$ is the sum of the interior angles of $\Delta$.

In the following, we are going to prove that the characterization of the scalar curvature via the local extent holds at normal points of a $CBB(k)$ surface. 
\begin{proposition}\label{pr_equal-Machi}
Let $(S,d)$ be a $CBB(k)$ surface, $k\in\R$. Then for $\mathcal H^2$-a.e. $x\in S$ it holds
\begin{equation}\label{equal-Machi}
G(x)=\frac{1}{2}\scal^{(2)}_d(x)
\end{equation}
\end{proposition}
\begin{proof}
Since $CBB(1)$ surfaces are in particular $CBB(-1)$, up to a rescaling we can suppose without loss of generality that $k=-1$. According to a well-known result by Alexandrov, \cite[Section XII.2]{Al}, each point in $(S,d)$ has a neighborhood isometric to the boundary of a convex set $K$ in $\H^3$. By \cite[Proposition 6]{Ve-hyper} a.e. point in $\partial K$ is normal. At every normal point $q\in\partial M$ the following hold: for $\epsilon\ll 1$, there exists a biLipschitz homeomorphism $\Phi_\e$ from the metric ball $B_\epsilon(q)\subset\partial K$ onto an open set $\Phi_\e(B_\epsilon(q))=:U_\epsilon\subset S^2_{G(q)}$ contained in a smooth surface $S^2_{G(q)}$ of constant sectional curvature $G(q)$, with Lipschitz constants $L(\Phi)$ and $L(\Phi^{-1})$ smaller than $(1+h(\epsilon^2))$, where the function $h$ is such that $\e^{-2}h(\e)\to 0$ as $\epsilon\to 0$, \cite[Corollary 9 and Theorem 14]{Ve-hyper}.

In particular for every couple of points $x$ and $y$ in $(S,d)$ such that $d(x,q)\leq \e$ and $d(y,q)\leq \e$, one has $\Phi_\e(x),\Phi_\e(y)\in \bar B_{\e+\e h(\e)}\subset S^2_{G(q)}$. Accordingly, 
\begin{align}\label{biLip1}
\mathrm{xt}_3(\bar B_\e^{S}(q)) 
&=\frac 13 \sup_{(x_1,x_2,x_3)\in S^3}\sum_{1\leq i<q\leq 3}d(x_i,x_j)\\
&\leq
\frac 13 (1+h(\e))\sup_{(x_1,x_2,x_3)\in S^3}\sum_{1\leq i<q\leq 3}d_{S^2_{G(q)}}(\Phi_\e(x_i),\Phi_\e(x_j))\nonumber\\
&\leq 
(1+h(\e))\mathrm{xt}_3(\bar B_{\e+\e h(\e)}^{S^2_{G(q)}}), \nonumber
\end{align}
so that 
\begin{align}\label{biLip2}
\frac{12n^2}{\e^2}\left\{1-
\frac{\mathrm{xt}_{3}(\bar B_\e^S(q))}{\mathrm{xt}_{3}(\bar B_\e^{\R^n}(x))}\right\}
&\geq \frac{12n^2}{\e^2}\left\{1-
\frac{(1+h(\e))\mathrm{xt}_{3}(\bar B_\e^{S^2_{G(q)}})}{\mathrm{xt}_{3}(\bar B_\e^{\R^n}(x))}\right\}\\
&\geq \frac{12n^2}{\e^2}\left\{1-
\frac{\mathrm{xt}_{3}(\bar B_\e^{S^2_{G(q)}})}{\mathrm{xt}_{3}(\bar B_\e^{\R^n}(x))}-h(\e)\sqrt{\frac{2n}{n+1}}\right\}.\nonumber
\end{align}
%
%
In particular, by Definition \ref{def_scal}, $\scal_{S}^{(2)}(q)\geq 2G(q)$. Exchanging the roles of $S$ and $S_q$ in \eqref{biLip1} and \eqref{biLip2} gives the converse inequality $\scal_{S}^{(2)}(q)\leq 2G(q)$.
\end{proof}

\section{Towards scalar curvature measure}\label{sect_scalmeas}

Kapovitch, Lytchak and Petrunin recently proposed a notion of metric scalar curvature $\nu$ on a metric measured space $(X,d,\mu)$, called \textsl{metric-measure} curvature, $\mu$  being a positive Radon measure on $(X,d)$. For $x\in X$ define $b_r(x):=\mu(B_r(x))$ and introduce the deviation measure $v_r$ on $X$, absolutely continuous with respect to $\mu$, given by 
\[v_r=\left(1-\frac{b_r}{\vol(B^{\R^n}_r)}\right)\mu.\]
This is primarly motivated by the fact that, when $X=(M,g)$ is a Riemannian manifold endowed with Riemannian distance $d=d_g$ and volume $\mu=\vol_g$, according to \eqref{vol-exp} the one-parameter family  $\{v_r/r^2\}_{r>0}$ converges as a measure to $\frac{1}{6(n+2)}\scal_g d\vol_g $ as $r\to 0$. Accordingly, one expects the limit measure 
\[
\nu:=\lim_{r\to 0} v_r/r^2,
\]
whenever it exists, to be a good candidate to replace the scalar curvature on measure metric spaces. In \cite{KLP} it is in particular proven that the family of measure $v_r/r^2$ is uniformly bounded on surfaces with bounded integral surfaces, so that it exists a (a priori non unique) metric-measure curvature $\nu$. Note that, unfortunately, $\nu$ do not coincide in general with the natural intrinsic curvature measure of the surface; see \cite[Exemple 1.14]{KLP}. 

In this section we adapt the construction of \cite{KLP} to define an \textit{extent-curvature measure} and a \textit{boundary-extent-curvature-measure} which generalize the Scalar curvature (measure) on a Riemannian manifold using \eqref{eq_def-scal} or \eqref{eq_def-scal_bdy} instead of \eqref{vol-exp}. As it is the case for the mm-curvature measure, we will see that extent-curvature measures and boundary-extent-curvature measures exist on BIC surfaces.

In the following, let $(S,d)$ be a surface of bounded integral curvature, and consider its area measure coinciding with the two-dimensional Hausdorff measure $\mathcal H^2$.
For any fixed scale $r$, consider the deviation measure $e_r$ on $X$, absolutely continuous with respect to $\mathcal H^2$, given by 
\[e_r:=\left(1-\frac{\xtt(\bar B^d_r(x))}{\xtt(\bar B^{\R^n}_r)}\right)\mathcal H^2.\]
We will say that $(S,d)$ has \textsl{locally finite extent curvature} if the family of measures $\{e_r/r^2\}_{0<r\leq 1}$ is uniformly bounded. Whenever a limit $\lim_{r\to 0} e_r/r^2$ exists, it is called \textsl{extent curvature measure}. 
Similarly, we can define 
\[\partial e_r:=\left(1-\frac{\partial \xtt(\bar B^d_r(x))}{\partial\xtt(\bar B^{\R^n}_r)}\right)\mathcal H^2\]
and say that $(S,d)$ has \textsl{locally finite boundary extent curvature} if the family of measures $\{\partial e_r/r^2\}_{0<r\leq 1}$ is uniformly bounded. Whenever a limit $\lim_{r\to 0} \partial e_r/r^2$ exists, it is called \textsl{boundary extent curvature measure}. 

We have the following result.
\begin{theorem}\label{th_locfincurv}
Let $(S,d)$ be a surface with bounded integral curvature.
\begin{enumerate}
 \item $(S,d)$  has locally finite extent curvature.
 \item $(S,d)$  has locally finite boundary extent curvature.
\end{enumerate} 
\end{theorem}

\begin{proof}

The proof is essentially the same as the proof of \cite[Theorem 4.2]{KLP}. We will sketch the relevant changes for completeness.

The main ingredient is the analogue of \cite[Lemma 4.1]{KLP}
\begin{lemma}\label{lem_KLP}
There exists some $\delta_0>0$ with the following property. Let
$S$ be a surface with bounded integral curvature and let $\omega$ be its curvature measure. Let $x\in X$ be a point and let $r>0$ be such that $\bar B^d_r(x)$ is compact, $|\omega|(B^d_r(x))<\delta_0$ and $\bar B^d_r(x)\subset U$ when $U$ is an open subset of $S$ homeomorphic to $\R^2$. Then
\begin{equation}\label{eq_lemKLPxt}
\left|1-\frac{\xtt(\bar B^d_r(x))}{\xtt(\bar B^{\R^2}_r)}\right| \leq 3 |\omega|(B^d_r(x))
\end{equation}
and
\begin{equation}\label{eq_lemKLPbxt}
\left|1-\frac{\partial\xtt(\bar B^d_r(x))}{\partial\xtt(\bar B^{\R^2}_r)}\right| \leq 3 |\omega|(B^d_r(x)).
\end{equation}
\end{lemma}
\begin{proof}
To prove \eqref{eq_lemKLPxt}, it suffices to prove that for all $0<s<r$ \[
\left|1-\frac{\xtt(\bar B^d_s(x))}{\xtt(\bar B^{\R^2}_s)}\right| \leq  3|\omega|(B^d_r(x)).
\]
Moreover it is enough to check this latter inequality on polyhedral metrics homeomorphic to $\R^2$. In fact, let $s<s'<t$. According to \cite[Theorem 4.3]{Re} we take a polyhedral sequence $\{d_k\}$ of metrics on $S$ approaching $d$ uniformly and tamely, see Theorm 8.4.3 and the subsequent discussion \cite{Re}, or \cite[Theorem 7 of Chapter VII]{AZ}. Without loss of generality let 
\[\sup_{y,y'\in U}|d(y,y')-d_k(y,y')|<\frac{1}{k}.\]
Suppose that
\[
\left|1-\frac{\xtt(\bar B^{d_k}_s(x))}{\xtt(\bar B^{\R^n}_s)}\right| \leq  3|\omega|(B^{d_k}_{s'}(x))
\]
for any polyhedral metric $d_k$. Let $x_{j,k}$, $j=1,2,3$ be a $3$-extender of $\bar B^{d_k}_s(x)$. Then 
 $x_{j,k}\in B^d_{s+\frac 1k}(x)$ and $\xtt(\bar B^{d}_s(x))\geq \limsup_{k\to \infty} \xtt(\bar B^{d_k}_s(x))$. Similarly, if $\{y_j\}_{j=1,2,3}$ is a $3$-extender of $\bar B^{d}_s(x)$, then $y_j\in B^{d_k}_{s+1/k}(x)$ so that $\xtt(\bar B^{d}_s(x))\leq \liminf_{k\to \infty} \xtt(\bar B^{d_k}_{s+1/k}(x))$.
 
Then the proof is the same as in \cite{KLP}, and it exploits the construction of an explicit completion of the local polyhedral surface to a global polyhedral BIC metric on $\R^2$ with controlled curvature, as well as the biLipschitz maps to $\R^2$ provided by \cite{BL-MathAnn}.

To conclude, one has just to remark that on regions $(1+\delta')$-biLipschitz diffeomorphic to Euclidean open sets it holds
\[
(1+\delta')^{-2}\leq \frac{\xtt(\bar B_r(x))}{\xtt(\bar B^{\R^n}_r)} \leq (1+\delta')^{2} 
\]
for $r$ small enough (compare with \cite[Section 2.4]{KLP}).

\medskip
The proof of \eqref{eq_lemKLPbxt} is essentially the same. One has just to keep in account that by definition $|\partial\xtt(\bar B_r(x))-\xtt(\bar B_r(x))|<\eta$ whenever the three points realizing $\xtt(\bar B_r(x))$ are in $B_r(x)\setminus B_{r-\eta/6}(x)$.
\end{proof}

We prove (1), the proof of (2) being the same. The result is local so that it suffices to prove it in a neighborhood of an arbitrary point $x_0\in S$. Let $X$ be a neighborhood of $x_0$ homeomorphic to $\R^2$ and satisfying $|\omega|(X\setminus\{x_0\})<\min\{\delta_0;1/3\}$, $\delta_0$ being the positive constant in Lemma \ref{lem_KLP}. Let $A\subset X$ compact. We want to prove that there exist an $\e>0$ and a constant $C=C(\epsilon)>0$ such that for all $0<r<\epsilon$ one has 
\begin{equation}
|e_r|(A)\leq Cr^2
\end{equation}
First note that for any $x\in S$ and positive $s$, it holds $\xtt(\bar B_s(x))\leq 2s$ and
\[ 1\geq 
1-\frac{\xtt(\bar B_s(x))}{\xtt(\bar B^{\R^n}_s)} \geq (1-2\sqrt {3}/3)>-1.\]
According to \cite{KLP}, up to take a smaller $\epsilon$ we have that for all $0<3r<\epsilon$
\[\mathcal H^2 (B_{3r}(x_0))< \frac {9r^2}{4\epsilon}.\]
Then 
\[
|e_r|(A\cap B_{2r}(x_0))\leq |e_r|( B_{2r}(x_0))\leq \mathcal H^2 (B_{3r}(x_0))< \frac {9r^2}{4\epsilon}.\]

On the other hand, for $x\in X\setminus B_r(x_0)$ we have $|\omega|(B_r(x))<\delta_0$. We can thus apply Lemma \ref{lem_KLP} and \cite[Lemmas 2.1 and 4.1]{KLP} to get
\begin{align*}
|e_r|(A\setminus B_{2r}(x_0)) &\leq \int_{A\setminus B_{2r}(x_0)} 3|\omega|(B_r(x)) d\mathcal H^2(x)\\&\leq 3\int_{A\setminus B_{r}(x_0)} \mathcal H^2(B_r(x)) d|\omega|(x)\\
&\leq 6\pi r^2 |\omega|(A\setminus B_{r}(x_0))\\
&\leq 6\pi\delta_0 r^2.
\end{align*}
Thus $|e_r|(A)\leq (6\pi\delta_0+9/4\epsilon) r^2$, as claimed.
\end{proof}

\bigskip
\textbf{Acknowledgement.} I would like to thank J. Bertrand for useful discussion on the subject. In particular, 
I am indebted to him for suggesting the approach to Proposition \ref{pr_equal-Machi} via the local isometry between Alexandrov surfaces and convex surfaces in spaces of constant curvature, and for pointing me out the Alexandrov's theorem on the a.e. second differentiability of convex functions. Also, I'm grateful to A. Bernig for pointing out to me the paper \cite{KLP}, from which Section \ref{sect_scalmeas} arose.

This research has been conducted as part of the project Labex MME-DII (ANR11-LBX-0023-01). The author is members of the ``Gruppo Nazionale per l'Analisi Matematica, la Probabilit\'a e le loro Applicazioni'' (GNAMPA) of the \textsl{Istituto Nazionale di Alta Matematica} (INdAM).

\bibliographystyle{alpha}
\bibliography{scalar}
%
%
%
%
%
%

\end{document}